\documentclass[11pt, reqno]{amsart}

\usepackage{amsmath}
\usepackage{amssymb}
\usepackage{amsthm}
\usepackage{mathrsfs}
\usepackage{mathtools}
\usepackage{bm}
\usepackage{enumerate}
\usepackage{comment}
\usepackage{mleftright}
\usepackage{color}
\usepackage{iftex}

\ifuptex
	\usepackage[dvipdfmx, hypertexnames=false, hidelinks]{hyperref}
	\hypersetup{colorlinks=true,citecolor=blue,linkcolor=blue,urlcolor=blue,
		pdfstartview=FitH }
\else
\ifpdftex
\usepackage
	{hyperref}
	\hypersetup{colorlinks=true,citecolor=blue,linkcolor=blue,urlcolor=blue,
		pdfstartview=FitH }
\fi\fi

\usepackage{cleveref}
\usepackage{autonum}

\allowdisplaybreaks[3]
\theoremstyle{plain} 
\newtheorem{theorem}{Theorem}
\newtheorem{lemma}{Lemma}[section]

\newtheorem{proposition}{Proposition}

\newtheorem*{conjecture*}{Conjecture}
\newtheorem*{theorem*}{Theorem}
\newtheorem*{question*}{Question}

\theoremstyle{plain}

\theoremstyle{remark}
\newtheorem{remark}{Remark}

\theoremstyle{definition}
\newtheorem*{assumption*}{Assumption}

\newtheorem*{notations*}{Notations}
\newtheorem*{acknowledgment*}{Acknowledgments}

\numberwithin{equation}{section}

\crefname{section}{Section}{Sections}

\crefname{theorem}{Theorem}{Theorems}
\crefname{corollary}{Corollary}{Corollaries}
\crefname{lemma}{Lemma}{Lemmas}
\crefname{proposition}{Proposition}{Propositions}
\crefname{claim}{Claim}{Claims}
\crefname{definition}{Definition}{Definitions}
\crefname{notation}{Notation}{Notations}
\crefname{problem}{Problem}{Problems}
\crefname{note}{Note}{Notes}
\crefname{remark}{Remark}{Remarks}
\crefname{example}{Example}{Examples}
\crefname{equation}{}{}

\crefname{enumi}{}{}
\crefname{enumii}{}{}
\crefname{enumiii}{}{}

\newcommand\swapcommand[2]{%
\let\swaptemp#1
\let#1#2
\let#2\swaptemp
}

\let\sl\l
\renewcommand\l{%
		\leavevmode
	\ifmmode
		\mleft
	\else
		\sl
	\fi
}

\let\sL\L
\renewcommand\L{%
		\leavevmode
	\ifmmode
	\mathscr{L}
	\else
		\sL
	\fi
}

\let\rmi\i
\renewcommand\i{%
		\leavevmode
	\ifmmode
	\mathrm{i}
	\else
		\rmi
	\fi
}

\newcommand\set[2]{%
	\left\{ #1:  #2 \right\}
}

\newcommand{\CC}{\mathbb{C}}
\newcommand{\RR}{\mathbb{R}}

\newcommand{\s}{\sigma}

\newcommand{\lam}{\lambda}
\newcommand{\Lam}{\Lambda}
\newcommand{\vp}{\varphi}

\newcommand{\ceq}{\coloneqq}
\newcommand{\eqc}{\eqqcolon}
\newcommand{\sd}{\, \mathrm{d}}
\renewcommand{\sL}{\mathscr{L}}

\newcommand{\ds}{\displaystyle}

\renewcommand{\a}{\alpha}
\renewcommand{\b}{\beta}
\renewcommand{\k}{\kappa}
\renewcommand{\r}{\mright}
\renewcommand{\d}{\mathrm{d}}

\renewcommand{\Re}{\operatorname{Re}}
\renewcommand{\Im}{\operatorname{Im}}
\renewcommand{\epsilon}{\varepsilon}

\renewcommand{\hat}{\widehat}
\renewcommand{\bar}{\overline}

\title{Explicit extreme values of the argument of the Riemann zeta-function}
\author[S. Inoue, H. Kobayashi, Y. Toma]{Sh\={o}ta Inoue, Hirotaka Kobayashi, and Yuichiro Toma}

\address[S. Inoue]{Department of Liberal Arts and Basic Sciences, College of Industrial Technology, Nihon University, 2-11-1 Shin-ei, Narashino, Chiba 275-8576, Japan}
\email{inoue.shota@nihon-u.ac.jp}

\address[H. Kobayashi]{National Fisheries University, 2-7-1, Nagatahon-machi, Shimonoseki-shi, Yamaguchi 759-6595, Japan}
\email{h.kobayashi@fish-u.ac.jp}

\address[Y. Toma]{Global Education Center, Waseda University, 1-6-1 Nishiwaseda, Shinjuku-ku, Tokyo 169-8050, Japan}
\email{yuichiro.toma@aoni.waseda.jp}



\begin{document}

\begin{abstract}
	We investigate explicit extreme values of the argument of the Riemann zeta-function in short intervals.
	As an application, we improve the result of Conrey and Turnage-Butterbaugh concerning $r$-gaps between zeros of the Riemann zeta-function.
\end{abstract}

\maketitle

\section{\textbf{Introduction and statement of results}}
	The argument of the Riemann zeta-function $\zeta$ on the critical line, usually denoted as $S(t)$,
	is a fascinating and intricate aspect of one of the most celebrated functions in number theory.
	The function is defined by $S(t) \ceq \frac{1}{\pi}\arg{\zeta(\frac{1}{2} + \i t)} = \frac{1}{\pi}\Im \int_{\infty}^{1/2}\frac{\zeta'}{\zeta}(\a + \i t)\sd\a$
	if $t$ is equal to neither $0$ nor the imaginary part of a zero of $\zeta$.
	If $t$ is equal to $0$ or the imaginary part of a zero, then $S(t) = (S(t + 0) + S(t - 0)) / 2$.
	By the argument principle, this function is influenced by the distribution of zeros of the Riemann zeta-function.
	The relationship is visualized by the Riemann-von Mangoldt formula:
	\begin{align}
		N(T)
		= \frac{T}{2\pi}\log\l( \frac{T}{2\pi e} \r) + \frac{7}{8} + S(T) + O\l( \frac{1}{T} \r),
	\end{align}
	where $N(T)$ denotes the number of zeros $\rho = \b + \i \gamma$ satisfying $0 < \gamma < T$ of $\zeta$ counted with multiplicity.
	If $T$ is equal to the imaginary part of a zero, $N(T) = (N(T + 0) + N(T - 0)) / 2$.

	In this paper, we discuss extreme values of $S(t + h) - S(t)$.
	The values of $S(t + h) - S(t)$ capture information on the number of zeros, as expressed by
	\begin{align}	\label{RvMFh}
		N(T + h) - N(T)
		= \frac{h}{2\pi}\log{T} + S(T + h) - S(T) + O\l( h + \frac{1}{T} \r).
	\end{align}
	Since detailed information about the zeros of the Riemann zeta-function has rich applications to the prime numbers, the study of $S(t + h) - S(t)$ is therefore also important.
	For this object, Selberg showed in an unpublished work that there exists a positive number $c = c(a, b)$, depending on arbitrary absolute positive constants $a, b$,
	such that for any large $T$ and any $h \in \l[a (\log{T})^{-1}, b(\log\log{T})^{-1}\r]$,
	\begin{align}	\label{STEV}
		\sup_{t \in [T, 2T]} \l\{ \pm\l( S(t + h) - S(t) \r) \r\}
		\geq c (h \log{T})^{\eta}
	\end{align}
	holds with $\eta = 1/2$ under the Riemann Hypothesis (RH).
	Later, Tsang gave a proof of this result in \cite{Ts1986} and also proved that \cref{STEV} holds with $\eta = 1/3$ unconditionally.
	More recently, the first author in \cite{I2024} used his results to study the distribution of zeros.
	In the paper, under RH, the first author also considered the explicit extreme values of $S(t + h) - S(t)$ and showed that $c = c(a, b)$ in inequality \cref{STEV} can be calculated by (4.1) in \cite[\S 4]{I2024},
	which is essentially obtained by following the argument of Selberg/Tsang straightforwardly.
	With this approach, one inevitably has $c < 1 / \sqrt{2 e \pi}$ when $a$ is large, and $b$ is small.
	The aim of this paper is to improve the explicit extreme values by using the method of Montgomery-Odlyzko \cite{MO}.
	The first result is the following.

	\begin{theorem}	\label{ORSi}
		Assume RH.
		For any large $T$ and any $h \in [C / \log{T}, c / \log\log{T}]$ with positive constants $C$ large and $c$ small,
		we have
		\begin{align}
			\sup_{T \leq t \leq 2T}\{ \pm(S(t + h) - S(t)) \}
			\geq \l( 1 - E \r)\sqrt{ \frac{h}{\pi} \log{T}},
		\end{align}
		where the error term $E$ satisfies
		\begin{align}
			E
			\ll \sqrt{h \log\log{T}} + \min\l\{\sqrt{\frac{\log^{3}(h \log{T})}{h \log{T}}}, \frac{(\log\log{T})^{3/2}}{h^{3/2} \log{T}}\r\}.
		\end{align}
	\end{theorem}

	This theorem can be applied to evaluate gaps of zeros of the Riemann zeta-function.
	Let $0 < \gamma_{1} \leq \gamma_{2} \leq \cdots \leq \gamma_{n} \leq \cdots$ denote the sequence of ordinates of the zeros of $\zeta$ in the upper half plane.
	We define the normalized large/small $r$-gap of nontrivial zeros by
	\begin{align}
		\lam_{r} = \limsup_{n \to + \infty}\frac{\gamma_{n + r} - \gamma_{n}}{2\pi r / \log\gamma_{n}},
		\quad
		\mu_{r} = \liminf_{n \to + \infty}\frac{\gamma_{n + r} - \gamma_{n}}{2\pi r / \log\gamma_{n}}.
	\end{align}
	From the Riemann-von Mangoldt formula, we have the trivial bounds $\mu_{r} \leq 1 \leq \lambda_{r}$.
	The nontrivial bounds in the case $r = 1$ have been studied by many mathematicians.
	The current best bounds under RH are $\lam_{1} > 3.18$ by Bui-Milinovich \cite{BM2017} and $\mu_{1} <0.515396$ by Preobrazhenski\u{\rmi} \cite{Pr2016}.
	For general $r$, Selberg \cite[p.355]{SCP} announced the nontrivial bounds of $\lam_{r}$, $\mu_{r}$ of the form
	\begin{align}	\label{r-gap}
		\lam_{r} \geq 1 + \frac{\Theta}{r^{\a}}, \quad \mu_{r} \leq 1 - \frac{\vartheta}{r^{\a}}
	\end{align}
	for all positive integer $r$.
	The numbers $\Theta$, $\vartheta$ which may depend on $r$ are greater than some absolute positive constants.
	Here, we may take $\a$ as $2/3$ unconditionally, and as $1/2$ under RH.
	Recently, Conrey and Turnage-Butterbaugh~\cite{CTB2018} proved an explicit result for the conditional bound.
	Specifically, they showed that \cref{r-gap} holds for $\Theta = 0.599648$ and $\vartheta = 0.379674$ with $\a = 1/2$ uniformly for $r \geq 1$ under RH.
	These results have been improved to $\Theta = A_{0} \ceq \max_{B > 0}\frac{2B}{\pi}\arctan\l( \frac{\pi}{B^{2}} \r) = 0.9064997 \cdots$, and $\vartheta = 0.484604$ in \cite{I2024}.
	Moreover, Conrey and Turnage-Butterbaugh proved that $\Theta = \vartheta = A_{0} + o(1)$ as $r \rightarrow + \infty$.
	As a consequence of \cref{ORSi} combined with the Riemann-von Mangoldt formula, we can improve the constant $A_{0}$ to $\sqrt{2} = 1.4142\dots$.

	\begin{theorem}	\label{LGZ}
		Assume RH.
		For any sufficiently large $r$, we have
		\begin{align}	\label{INELGZ}
			\lam_{r} \geq 1 + \frac{\sqrt{2}}{\sqrt{r}} - C_{1} \frac{(\log{r})^{3/2}}{r}, \quad
			\mu_{r} \leq 1 - \frac{\sqrt{2}}{\sqrt{r}} + C_{2} \frac{(\log{r})^{3/2}}{r}.
		\end{align}
		Here, $C_{1}$ and $C_{2}$ are some absolute positive constants.
	\end{theorem}

	To prove \cref{ORSi}, we employ the resonance method established by Soundararajan \cite{So2008}.
	From the celebrated works due to Soundararajan~\cite{So2008}, and Bondarenko and Seip~\cite{BS2017},
	the resonance method is now widely regarded as a powerful tool for detecting extreme values of number-theoretic objects.
	Inspired by these studies, we apply the resonance method to $S(t + h) - S(t)$.
	The method due to Montgomery and Odlyzko is a precursor of our approach. In fact, by combining our method using the resonance method with the Riemann-von Mangoldt formula, we can recover the original result of Montgomery and Odlyzko.

	In~\cite{MO}, Montgomery and Odlyzko showed that $\lambda_1>1.9799$ and $\mu_1<0.5179$.
	For a large real number $T$ and $L \leq T/(\log T)^2$, they studied the function $\tau$ defined by
	\begin{align}
		\tau(\xi; f)
		\ceq \xi - \l(\Re \frac{2}{\pi}\sum_{k m \leq L} \frac{\Lam(k)}{\sqrt{k} \log{k}} \sin(\pi \xi \tfrac{\log{k}}{\log T}) f(m)\bar{f(k m)} \bigg/ \sum_{n \leq L} |f(n)|^2 \r),
	\end{align}
	where $\xi$ is a positive number, and $f$ is a certain arithmetic function.
	They showed (for the case $r=1$ and extended $r \geq 2$ by Conrey and Turnage-Butterbaugh~\cite{CTB2018})
	that if there exists $\xi_r$ such that $\tau(\xi_{r}; f) < r$, then $\lambda_{r} \geq \xi_{r}$, and if there exists $\xi_r$ such that $\tau(\xi_{r}; f) > r$, then $\mu_{r} \leq \xi_{r}$.

	To conclude this section, we give a limitation of the method of Montgomery and Odlyzko.

	\begin{theorem}	\label{LGZMO}
		Let $f$ be an arithmetic function not identically zero, let $L$ be large, and let $h > 0$.
		For any $W > 0$, we have
		\begin{align}	\label{LGZMO1}
			&\bigg|\Re\frac{2}{\pi} \sum_{km \leq L} \frac{\Lam(k)}{\sqrt{k} \log{k}} \sin(\tfrac{h}{2}\log{k}) f(m) \bar{f(k m)}\bigg| \bigg/ \sum_{n \leq L}|f(n)|^{2}\\
			&\leq \max_{1 \leq l \leq L} \l\{\frac{\sqrt{W}}{2} \vp\l( \tfrac{h}{2\pi} \log(L / l) \r) + \frac{h \log{l}}{\pi\sqrt{W}}\r\} + O(h \sqrt{W}).
		\end{align}
		Here, $\vp(x) = \int_{0}^{x}(\sin(\pi u) / \pi u)^{2}\sd u$.
		In particular, we have
		\begin{align}
			|\tau(\xi; f) - \xi|
			\leq \max_{1 \leq l \leq L} \l\{\frac{\sqrt{W}}{2} \vp\l( \xi \frac{\log(L / l)}{\log T} \r) + \frac{2 \xi}{\sqrt{W}}\frac{\log{l}}{\log T}\r\}
			+ O\l(\frac{\xi}{\log T} \sqrt{W}\r)
		\end{align}
		for any $\xi > 0$, any large $L, T$, and any $W > 0$.
	\end{theorem}
	By this theorem together with a numerical calculation, we obtain that the limitations of large/small gaps of zeros in the method of Montgomery-Odlyzko are $\lam_{1} \geq 3.022$, $\mu_{1} \leq 0.508$ when $L \leq T$.
	There are previous works on such limitations in the method \cite{CGG1983}, \cite{GTTT2023}.
	Conrey, Ghosh, and Gonek \cite{CGG1983} showed that $\lam_{1} \geq 3.74$ and $\mu_{1} < 1/2$ are limitations of the Montgomery-Odlyzko method.
	Moreover, the recent work by Goldston, Trudgian, and Turnage-Butterbaugh \cite{GTTT2023} improved on the result on their limitation of small gaps to $\mu_{1} \leq 0.5042$.
	They also remarked that their method can be applied to the bound of $\lam_{1}$, which leads us to $\lam_{1} \geq 3.6747$.
	Our result gives an improvement on those results.
	On the other hand, Bui and Milinovich~\cite{BM2017} applied Hall's method \cite{Ha2005} to prove $\lambda_1>3.18$ under RH.
	Our limitation for $\lam_{1}$ shows that the result of Bui-Milinovich goes beyond the barrier imposed by the method of Montgomery-Odlyzko.
	Furthermore, we can also see that $\lam_{r} \geq 1 + \sqrt{2/r} - O(1/r)$, $\mu_{r} \leq 1 - \sqrt{2/r} + O(1 / r)$ are limitations of their method for $r$-gaps of zeros when $L \leq T$.
	This observation shows that the constant $\sqrt{2}$ in \cref{LGZ} is optimal.

	This paper is organized as follows.
	In \cref{Sec:RStGZ} we discuss the relationship between large values of $S(t)$ in short intervals and the gaps between consecutive $r$ zeros of the Riemann zeta-function.
	In Section \ref{Resonance method}, we apply the resonance method to $S(t)$ in short intervals.
	Combining this result and Proposition~\ref{ELBRTsiRktrk}, we prove Theorem~\ref{ORSi} in Section~\ref{Proof of Theorem 1}.
	In Section~\ref{Proof of Theorem 2}, we prove Theorem~\ref{LGZ} by using Theorem~\ref{ORSi} and the relationship between $S(t)$ and gaps of zeros established in \cref{Sec:RStGZ}.
	In \cref{Sec:Proof_UBMO}, we prove \cref{LGZMO}, and finally in \cref{Sec:limitation}, we derive the resulting limitations on large and small gaps between zeros that follow from \cref{LGZMO}.

\section{\textbf{A relationship between $S(t)$ and large/small gaps of consecutive $r$ zeros}}	\label{Sec:RStGZ}

	By the same strategy as in the proof of Theorem~1 in \cite{I2024}, we obtain the following relation between $S(t)$ and gaps of zeros.
	\begin{proposition}	\label{RStGZ}
		Let $r$ be a positive integer, and let $\theta$ be a positive number which may depend on $r$.
		Then the inequality $\lambda_r > \theta$ holds if and only if
		there exist numbers $b > 0$, $\theta' > 1$ and a sequence $\{T_{n}\}$ satisfying $\theta' > \theta$, $b > r(\theta' - 1)$, and $T_{n} \rightarrow +\infty$ as $n \rightarrow + \infty$ such that
		\begin{align}
			\inf_{t \in [T_{n}, 2T_{n}]} \l\{S\l(t + 2 \pi r \theta' / \log T_{n} \r) - S(t) \r\}
			\leq -b.
		\end{align}
		Similarly, the inequality $\mu_{r} < \theta$ holds if and only if
		there exist numbers $b > 0$, $0 < \theta' < 1$ and a sequence $\{T_{n}\}$ satisfying $\theta' < \theta$, $b > r(1 - \theta')$, and $T_{n} \rightarrow +\infty$ as $n \rightarrow + \infty$ such that
		\begin{align}
			\sup_{t \in [T_{n}, 2T_{n}]} \l\{S\l(t + 2 \pi r \theta' / \log T_{n} \r) - S(t) \r\}
			\geq b.
		\end{align}
	\end{proposition}

	\begin{proof}
		Since the first and second assertion can be proved by the same argument, we only give the proof of the first assertion.
		We use the simple equivalence which is that, for any $\{ T_{n} \}$ satisfying $T_{n} \rightarrow +\infty$ as $n \rightarrow + \infty$, there exists some $t \in [T_{n}, 2T_{n}]$ such that
		\begin{align}	\label{EIE}
			N(t + h) - N(t) < r
		\end{align}
		if and only if the inequality
		\begin{align}	\label{EIE2}
			\sup_{\gamma_{m}, \gamma_{m + r} \in [T_{n}, 2T_{n} + h]}\frac{\gamma_{m + r} - \gamma_{m}}{h} > 1
		\end{align}
		holds.

		First, we assume $\lam_{r} > \theta$.
		Then there exist a number $\theta'$ and a sequence $\{T_{n}\}$ satisfying $\theta' > \theta$ and $T_{n} \rightarrow +\infty$ as $n \rightarrow + \infty$ such that
		\begin{align}
			\sup_{\gamma_{m}, \gamma_{m + r} \in [T_{n}, 2T_{n}]}\frac{\gamma_{m + r} - \gamma_{m}}{2\pi r \theta' / \log T_{n}}
			\geq \sup_{\gamma_{m}, \gamma_{m + r} \in [T_{n}, 2T_{n}]}\frac{\gamma_{m + r} - \gamma_{m}}{2\pi r \theta' / \log(\gamma_{m} / 2)}
			> 1
		\end{align}
		holds for any sufficiently large $n$.
		Therefore, \cref{EIE2} holds when $h = 2 \pi r \theta' / \log T_{n}$ and $n$ is sufficiently large.
		Hence, there exists a $t \in [T_{n}, 2T_{n}]$ such that \cref{EIE} holds with $h = 2 \pi r \theta' / \log T_{n}$,
		which is also equivalent to
		\begin{align}
			N(t + 2 \pi r \theta' / \log T_{n}) - N(t) \leq r - 1 / 2.
		\end{align}
		Combining this with \cref{RvMFh}, we have
		\begin{align}
			\inf_{t \in [T_{n}, 2T_{n}]} (S(t + 2 \pi r \theta' / \log T_{n}) - S(t))
			&\leq \inf_{t \in [T_{n}, 2T_{n}]} (S(t + 2 \pi r \theta' / \log T_{n}) - S(t))\\
			&\leq r - 1 / 2 - r \theta' + o(1)
			\leq -b
		\end{align}
		with $b = r(\theta' - 1) + 1 / 3$.

		Next, we assume that there exist numbers $b, \theta'$ and a sequence $\{T_{n}\}$ satisfying $\theta' > \theta$, $b > r(\theta' - 1)$, and $T_{n} \rightarrow +\infty$ as $n \rightarrow + \infty$ such that
		\begin{align}
			\inf_{t \in [T_{n}, 2T_{n}]} \l(S\l(t + 2 \pi r \theta' / \log T_{n} \r) - S(t) \r) \leq -b
		\end{align}
		for any sufficiently large $n$.
		Then, it holds by \cref{RvMFh} that for any large $n$
		\begin{align}
			N(t + 2 \pi r \theta' / \log T_{n}) - N(t)
			\leq r \theta' \frac{\log{t}}{\log T_{n}} - b + o(1)
			\leq r - (b - r(\theta' - 1)) + o(1)
		\end{align}
		for some $t \in [T_{n}, 2T_{n}]$.
		Therefore, \cref{EIE} holds when $h = 2 \pi r \theta' / \log T_{n}$ and $n$ is sufficiently large.
		Hence, we find that
		\begin{align}
			\lam_{r}
			&= \limsup_{m \rightarrow + \infty}\frac{\gamma_{m + r} - \gamma_{m}}{2\pi r / \log{\gamma_{m}}}
			\geq \lim_{n \rightarrow + \infty}\sup_{\gamma_{m}, \gamma_{m + r} \in [T_{n}, 2T_{n} + h]}
			\frac{\gamma_{m + r} - \gamma_{m}}{2\pi r / \log{\gamma_{m}}}\\
			&= \lim_{n \rightarrow + \infty}\sup_{\gamma_{m}, \gamma_{m + r} \in [T_{n}, 2T_{n} + h]}
			\frac{\gamma_{m + r} - \gamma_{m}}{h}\frac{h}{2\pi r / \log{T_{n}}}\frac{\log{\gamma_{m}}}{\log{T_{n}}}
			\geq \theta'
			> \theta,
		\end{align}
		which completes the proof of \cref{RStGZ}.
	\end{proof}

	\begin{remark}
		In \cref{RStGZ}, if we change the interval $[T_{n}, 2T_{n}]$ to $[T_{n}^{a}, 2T_{n}]$ for some $0 < a < 1$, then the equivalence no longer holds.
		Although the statement can be suitably modified, the resulting inequalities for $\lam_{r}$, $\mu_{r}$ become weaker than the original form.
		Hence, we consider the extreme value of $S(t + h) - S(t)$ over the interval $[T, 2T]$ in \cref{ORSi}.
	\end{remark}

\section{\textbf{Resonance method}}\label{Resonance method}
	The resonance method aims to extract the large values of an objective function
	by comparing the mean value of the objective function multiplied by a ``resonator'' with that of the resonator itself.
	In this paper, we take the resonator to be the Dirichlet polynomial
	\begin{align}
		R(t) = \sum_{n \leq L}f(n)n^{-\i t},
	\end{align}
	following the works \cite{BS2017}, \cite{BS2018}, \cite{MO}, and \cite{So2008}.
	Here, the arithmetic function $f$ is chosen suitably depending on the objective function.
	In this section, we evaluate the extreme value of $S(t + h) - S(t)$ by means of general forms of resonators.
	We construct a suitable resonator for our purpose in \cref{Proof of Theorem 1}.

	In this section, we aim to prove the following proposition.

	\begin{proposition}	\label{LBStvR}
		Assume RH.
		For any arithmetic function $f$ that is not identically zero, any large $L, T$ satisfying $L \leq T / (\log{T})^{2}$, and any $h > 0$ we have
		\begin{align}
			&\sup_{T \leq t \leq 2T}\l\{\pm\l(S(t + h) - S(t) \r)\r\}\\
			&\geq \mp\l( 1 + O\l(\frac{1}{T}\r) \r)\frac{2}{\pi} \Re \sum_{k m \leq L} \frac{\Lam(k)}{\sqrt{k} \log{k}} \sin(\tfrac{h}{2}\log{k}) f(m)\bar{f(k m)} \bigg/ \sum_{n \leq L} |f(n)|^{2}
			+ O\l( \frac{1}{T} \r).
		\end{align}
		Here, the implicit constant is absolute.
	\end{proposition}

	\subsection{Preliminaries}
		Throughout this paper, we set $\Phi(t) = e^{-t^{2}/2}$.
		As an auxiliary result, we first prove the following proposition.

		\begin{proposition}	\label{EFCVStRtPhi}
			Assume RH.
			For any arithmetic function $f$, any $L, T \geq 3$ satisfying $L \leq T / (\log{T})^{2}$, and any $h > 0$ we have
			\begin{align}
				&\int_{-\infty}^{\infty}\l\{ S(t + \tfrac{h}{2}) - S(t - \tfrac{h}{2}) \r\} |R(t)|^{2} \Phi\l( \frac{t - 3T / 2}{T / \log{T}} \r)\sd t\\
				&= -\sqrt{2\pi}\frac{T}{\log{T}}\frac{2}{\pi}\Re\sum_{k m \leq L} \frac{\Lam(k)}{\sqrt{k} \log{k}} \sin(\tfrac{h}{2}\log{k}) f(m)\bar{f(k m)}
				+ O\l( \frac{1}{T}\sum_{n \leq L}|f(n)|^{2} \r).
			\end{align}
		\end{proposition}

		To show this proposition, we require some auxiliary lemmas.

		\begin{lemma}	\label{CscrM-INEsRt}
			For any arithmetic function $f$ and any $L \geq 3$ we have
			\begin{align}	\label{INEsRt}
				|R(t)|^{2}
				\leq \sum_{m, n \leq L} |f(m) f(n)|
				\leq L \sum_{n \leq L}|f(n)|^{2}.
			\end{align}
		\end{lemma}

		\begin{proof}
			The first inequality of \cref{INEsRt} is obvious by the triangle inequality.
			We also find by the Cauchy-Schwarz inequality that
			\begin{align}
				\sum_{m, n \leq L} |f(m) f(n)|
				= \l( \sum_{n \leq L}|f(n)| \r)^{2}
				\leq \sum_{n \leq L} 1 \times \sum_{n \leq L}|f(n)|^{2}
				= L \times \sum_{n \leq L}|f(n)|^{2}.
			\end{align}
			Hence, we obtain inequality \cref{INEsRt}.
		\end{proof}

		The following lemma gives an explicit bound for estimates shown in Lemma 5 of \cite{BS2018}.

		\begin{lemma}	\label{lem5MSMA}
			For any arithmetic function $f$ and any $L, T \geq 3$ satisfying $L \leq T / (\log{T})^{2}$ we have
			\begin{align}
				\int_{-\infty}^{\infty}|R(t)|^{2}\Phi\l( \frac{t - 3T/2}{T / \log{T}} \r)\sd t
				= \sqrt{2\pi} \frac{T}{\log{T}} \l( 1 + O\l( \frac{1}{T} \r) \r) \sum_{n \leq L} |f(n)|^{2}.
			\end{align}
		\end{lemma}

		\begin{proof}
			It holds from the definition of $R(t)$ that
			\begin{align}
				\int_{-\infty}^{\infty}|R(t)|^{2}\Phi\l( \frac{t - 3T / 2}{T / \log{T}} \r)\sd t
				= \sqrt{2\pi} \frac{T}{\log{T}} \sum_{m, n \leq L} f(m) \bar{f(n)} \l( \frac{m}{n} \r)^{-3 \i T / 2} \Phi\l( \frac{T}{\log{T}} \log{\frac{m}{n}} \r)
			\end{align}
			since $\int_{-\infty}^{\infty} \Phi(u) e^{-\i x u} \sd u = \sqrt{2\pi} \Phi(x)$.
			Using \cref{INEsRt}, we find that
			\begin{align}
				\l|\sum_{\substack{m, n \leq L\\ m\not= n}} f(m) \bar{f(n)} \l( \frac{m}{n} \r)^{3 \i T / 2} \Phi\l( \frac{T}{\log{T}} \log{\frac{m}{n}} \r)\r|
				&\leq \Phi\l(\frac{T}{\log{T}} \frac{(\log{T})^{2}}{2T}\r)\sum_{m, n \leq L}|f(m)f(n)|\\
				&\leq \Phi\l(\tfrac{1}{2}\log{T}\r) L \sum_{n \leq L}|f(n)|^{2}
				\leq \frac{1}{T} \sum_{n \leq L}|f(n)|^{2}.
			\end{align}
			Adding the diagonal-terms to this, we complete the proof of \cref{lem5MSMA}.
		\end{proof}

		\begin{lemma}	\label{lem_STF}
			Let $V$ be an analytic function in the horizontal strip $\set{z \in \CC}{-\frac{3}{2} \leq \Im z \leq 0}$ satisfying
			$
				\ds{\sup_{-\frac{3}{2} \leq y \leq 0}|V(x + \i y)| \ll (|x| \log^{2}{x})^{-1}}.
			$
			For any $v \in \RR$, we have
			\begin{align}
				&\int_{-\infty}^{\infty}\log\zeta(\tfrac{1}{2} + \i(t + v))V(t)\sd t\\
				&= \sum_{n = 2}^{\infty}\frac{\Lam(n)}{n^{\frac{1}{2} + \i v}\log{n}}\hat{V}\l( \frac{\log{n}}{2\pi} \r)
				+ 2\pi \sum_{\b > \frac{1}{2}}\int_{0}^{\b - \frac{1}{2}}V(\gamma - v - \i\s)\sd\s
				- 2\pi \int_{0}^{\frac{1}{2}}V(- v - \i\s)\sd\s.
			\end{align}
			Here, $\hat{V}$ is the Fourier transform of $V$ defined by $\hat{V}(z) = \int_{-\infty}^{\infty}V(x)e^{-2\pi \i x z}\sd x$.
		\end{lemma}

		\begin{proof}
			This is obtained by (2.14) in \cite{Ts1986} and the argument below (2.14).
		\end{proof}

	\subsection{Proof of \cref{EFCVStRtPhi}}

			We write
			\begin{align}
				&\int_{-\infty}^{\infty}\log\zeta(\tfrac{1}{2} + \i(t \pm \tfrac{h}{2})) |R(t)|^{2} \Phi\l( \frac{t - 3T / 2}{T / \log{T}} \r)\sd t\\
				&= \sum_{m, n \leq L}f(m)\bar{f(n)} \int_{-\infty}^{\infty}\log\zeta(\tfrac{1}{2} + \i(t \pm \tfrac{h}{2})) \l( \frac{m}{n} \r)^{-\i t} \Phi\l( \frac{t - 3T / 2}{T / \log{T}} \r)\sd t.
			\end{align}
			We use \cref{lem_STF} with $V_{m, n}(z) = (n / m)^{\i z} \Phi((z - 3T / 2) / (T / \log{T}))$ to find that this equals
			\begin{multline}
				\sum_{m, n \leq L}f(m) \bar{f(n)} \Biggl\{\sum_{k = 2}^{\infty}\frac{\Lam(k)}{k^{\frac{1}{2} \pm \i h / 2}\log{k}}\hat{V_{m, n}}\l( \frac{\log{k}}{2\pi} \r)\\
				- 2\pi\int_{0}^{\frac{1}{2}}\l( \frac{m}{n} \r)^{-\s \pm \i (h / 2)}\Phi\l( \frac{\mp (h / 2) - \i\s - 3T / 2}{T / \log{T}} \r) \sd \s \Biggr\}
			\end{multline}
			under RH. The latter term is
			\begin{align}
				\ll \sum_{m, n \leq L}|f(m) f(n)| \times L^{1/2} \Phi(\log{T})
				\ll \frac{1}{T} \sum_{n \leq L}|f(n)|^{2}
			\end{align}
			by \cref{INEsRt}.
			Also, the former term is
			\begin{align}
				&= \sqrt{2\pi}\frac{T}{\log{T}}\sum_{m, n \leq L} \sum_{k = 2}^{\infty} f(m) \bar{f(n)}
				\frac{\Lam(k)}{\sqrt{k} \log{k}} k^{\mp \i h / 2} \l( \frac{km}{n} \r)^{-3\i T / 2} \Phi\l( \frac{T}{\log{T}}\log\l( \frac{km}{n} \r) \r)
			\end{align}
			since $\int_{-\infty}^{\infty} \Phi(u) e^{-\i x u} \sd u = \sqrt{2\pi} \Phi(x)$.
			Simple calculations using \cref{INEsRt} show that
			\begin{align}
				&\underset{km \not= n}{\sum_{m, n \leq L} \sum_{2 \leq k \leq T^{2}}} |f(m) f(n)| \frac{1}{k^{1/2}} \Phi\l( \frac{T}{\log{T}}\log\l( \frac{km}{n} \r) \r)\\
				&\ll L \sum_{n \leq L} |f(n)|^{2} \sum_{2 \leq k \leq T^{2}} \frac{1}{k^{1/2}} \Phi(\log{T})
				\ll \frac{1}{T^{2}} \sum_{n \leq L}|f(n)|^{2},
			\end{align}
			and that
			\begin{align}
				&\underset{km \not= n}{\sum_{m, n \leq L} \sum_{k > T^{2}}} |f(m) f(n)| \frac{1}{k^{1/2}} \Phi\l( \frac{T}{\log{T}}\log\l( \frac{km}{n} \r) \r)\\
				&\ll \sum_{m, n \leq L}|f(m) f(n)| \sum_{k > T^{2}} \frac{1}{k^{1/2}} k^{-T / 2\log{T}}
				\ll L \sum_{n \leq L} |f(n)|^{2} \times T^{-3}
				\ll \frac{1}{T^{2}} \sum_{n \leq L}|f(n)|^{2}.
			\end{align}
			Following these, we have
			\begin{align}
				&\int_{-\infty}^{\infty}\log\zeta(\tfrac{1}{2} + \i(t \pm \tfrac{h}{2})) |R(t)|^{2} \Phi\l( \frac{t - 3T / 2}{T / \log{T}} \r)\sd t\\
				&= \sqrt{2\pi}\frac{T}{\log{T}}\sum_{km \leq L} \frac{\Lam(k)}{\sqrt{k} \log{k}}k^{\mp \i h  /2} f(m)\bar{f(k m)}
				+ O\l( \frac{1}{T}\sum_{n \leq L}|f(n)|^{2} \r).
			\end{align}
			This also leads to
			\begin{align}
				&\int_{-\infty}^{\infty}\l\{ S(t + \tfrac{h}{2}) - S(t - \tfrac{h}{2}) \r\} |R(t)|^{2} \Phi\l( \frac{t - 3T / 2}{T / \log{T}} \r)\sd t\\
				&= -\sqrt{2\pi}\frac{T}{\log{T}}\frac{2}{\pi}\Re\sum_{km \leq L} \frac{\Lam(k)}{\sqrt{k} \log{k}} \sin(\tfrac{h}{2}\log{k}) f(m)\bar{f(k m)}
				+ O\l( \frac{1}{T}\sum_{n \leq L}|f(n)|^{2} \r).
			\end{align}
			Thus, we complete the proof of \cref{EFCVStRtPhi}.
			\qed

	\subsection{\textbf{Proof of \cref{LBStvR}}}\mbox{}
		Let $L, T$ be large numbers that satisfy $L \leq T / (\log{T})^{2}$, and let $h > 0$.
		Those parameters are chosen later.
		Write
		\begin{align}
			I = \int_{-\infty}^{\infty}\l\{S(t + \tfrac{h}{2}) - S(t - \tfrac{h}{2}) \r\} |R(t)|^{2} \Phi\l( \frac{t - 3 T / 2}{T / \log{T}} \r)\sd t.
		\end{align}
		We then find by \cref{EFCVStRtPhi} that
		\begin{align}	\label{ORS1UCp0}
			I
			= -\frac{\sqrt{2\pi} T}{\log{T}}\frac{2}{\pi}\Re\sum_{2 \leq k \leq L} \frac{\Lam(k)}{\sqrt{k} \log{k}} \sin(\tfrac{h}{2}\log{k})\sum_{k m \leq L}f(m)\bar{f(k m)}
			+ O\l( \frac{1}{T}\sum_{n \leq L}|f(n)|^{2} \r).
		\end{align}
		First, we show that
		\begin{align}
			I
			&= \int_{4T/3 \leq t \leq 5T/3}\l(S(t + \tfrac{h}{2}) - S(t - \tfrac{h}{2})\r) |R(t)|^{2} \Phi\l( \frac{t - 3T / 2}{T / \log{T}} \r)\sd t\\
			\label{ORS1UCp1}
			&\quad+ O\l( \frac{1}{T} \sum_{n \leq L}|f(n)|^{2} \r).
		\end{align}
		Using \cref{CscrM-INEsRt} and the estimate $S(t) \ll \log(|t| + 3)$, we find by simple calculations that
		\begin{align}
			\int_{t \leq 4T / 3} \l(S(t + \tfrac{h}{2}) - S(t - \tfrac{h}{2})\r) |R(t)|^{2} \Phi\l( \frac{t - 3T / 2}{T / \log{T}} \r)\sd t
			&\ll \frac{1}{T} \sum_{n \leq L}|f(n)|^{2},
		\end{align}
		and that
		\begin{align}
			\int_{t > 5T / 3}\l(S(t + \tfrac{h}{2}) - S(t - \tfrac{h}{2})\r) |R(t)|^{2} \Phi\l( \frac{t - 3T / 2}{T / \log{T}} \r)\sd t
			&\ll \frac{1}{T} \sum_{n \leq L}|f(n)|^{2}.
		\end{align}
		Therefore, we obtain \cref{ORS1UCp1}.

		We extract extreme values of $\pm\{S(t + h) - S(t)\}$ by
		\begin{align}
			&\pm \int_{4T/3 \leq t \leq 5T/3}\l(S(t + \tfrac{h}{2}) - S(t - \tfrac{h}{2})\r) |R(t)|^{2} \Phi\l( \frac{t - 3T / 2}{T / \log{T}} \r)\sd t\\
			&\leq \sup_{T \leq t \leq 2T}\l\{\pm \l(S(t + h) - S(t)\r)\r\}
			\int_{-\infty}^{\infty}|R(t)|^{2} \Phi\l( \frac{t - 3T / 2}{T / \log{T}} \r)\sd t\\
			&\leq \frac{\sqrt{2\pi} T}{\log{T}} \l( 1 + O\l( \frac{1}{T} \r) \r) \sup_{T \leq t \leq 2T}\l\{\pm \l(S(t + h) - S(t)\r)\r\} \sum_{n \leq L}|f(n)|^{2}.
		\end{align}
		In the last step, we have used \cref{lem5MSMA}.
		Combining this with \cref{ORS1UCp0}, we obtain
		\begin{align}
			&\sup_{T \leq t \leq 2T}\l\{\pm\l(S(t + h) - S(t)\r)\r\} \times \sum_{n \leq L} |f(n)|^{2}\\
			&\geq \mp\l( 1 + O\l(\frac{1}{T}\r) \r)\frac{2}{\pi} \Re\sum_{2 \leq k \leq L} \frac{\Lam(k)}{\sqrt{k} \log{k}} \sin(\tfrac{h}{2}\log{k}) \sum_{k m \leq L}f(m)\bar{f(k m)}
			+ O\l( \frac{1}{T}\sum_{n \leq L}|f(n)|^{2} \r).
		\end{align}
		This completes the proof of \cref{LBStvR}.
		\qed

\section{\textbf{Proof of \cref{ORSi}}}\label{Proof of Theorem 1}
	In this section, we let $L$ denote a large number, $h \in [C / \log{L}, c / \log\log{L}]$ with positive constants $C$ large and $c$ small.
	We choose $f = f_{\pm}$ as the multiplicative function supported on square-free numbers such that for any prime $p$
	\begin{align}
		f_{\pm}(p)
		\ceq \pm \sqrt{Q} \cdot \frac{\sin(\frac{h}{2}\log{p})}{p^{1/2 + \k h} h \log{p}}
	\end{align}
	if $\exp(\sqrt{\log\log{L}} / \sqrt{h}) \eqc M < p \leq L$ and $f_{\pm}(p) = 0$ otherwise.
	Here, the numbers $\k$ and $Q$ are to be chosen as
	\begin{align}
		Q = 4 \k (1 - y) h \log{L} \bigg/ \pi \int_{0}^{h \log{L} / 2\pi} \frac{\sin^{2}(\pi u)}{(\pi u)^{3}}\frac{e^{2 \k \pi u} - 1}{e^{4 \k \pi u}}\sd u, \quad
		\k = \frac{\log(h \log{L})}{y h \log{L}},
	\end{align}
	where
	$
		y = \sqrt{\log(h \log{L}) / h \log{L}}.
	$
	It then holds that $Q \asymp h \log{L}$ and $\k, y$ are sufficiently small when $h \in [C / \log{L}, c / \log\log{L}]$.
	These parameters are determined to optimize quantity \cref{p1ELBRTsiRktrk} below.
	For this $f_{\pm}$, we give a lower bound of the ratio of resonator in the following proposition.

	\begin{proposition}	\label{ELBRTsiRktrk}
		Let $L$ be large, and let $C / \log{L} \leq h \leq c / \log\log{L}$ with positive constants $C$ large and $c$ small.
		Then we have
		\begin{align}
			&\pm \frac{2}{\pi} \sum_{k m \leq L} \frac{\Lam(k)}{\sqrt{k} \log{k}} \sin(\tfrac{h}{2}\log{k}) f_{\pm}(m) f_{\pm}(k m) \bigg/ \sum_{n = 1}^{\infty}f_{\pm}(n)^{2}\\
			&\geq \l\{1 + O\l(\sqrt{h \log\log{L}} + \min\l\{ \sqrt{\frac{\log^{3}(h \log{L})}{h \log{L}}}, \frac{(\log\log{L})^{3/2}}{h^{3/2} \log{L}} \r\}\r) \r\}\sqrt{\frac{h}{\pi} \log{L}}.
		\end{align}
	\end{proposition}

	\cref{ORSi} immediately follows from \cref{LBStvR} and this proposition in the case $L = T / (\log T)^{2}$.

	\begin{proof}
		Put $\a = \k h$.
		First, we observe by the definition of $f_{\pm}$ that
		\begin{align}
			\label{ELBRTsip1Rankintrk}
			&\pm \frac{2}{\pi} \sum_{k m \leq L} \frac{\Lam(k)}{\sqrt{k} \log{k}} \sin(\tfrac{h}{2}\log{k}) f_{\pm}(m) f_{\pm}(k m)
			= \sqrt{Q} \frac{2}{\pi} \sum_{M < p \leq L} \frac{\sin^{2}(\tfrac{h}{2}\log{p})}{p^{1 + \a} h \log{p}} \sum_{\substack{m \leq L / p\\ p \nmid m}}f_{\pm}(m)^{2}.
		\end{align}
		We find by the definition of $f_{\pm}$ and Rankin's trick that
		\begin{align}
			\sum_{\substack{m \leq L / p\\ p \nmid m}}f_{\pm}(m)^{2}
			&\geq \sum_{\substack{n = 1\\ p \nmid n}}^{\infty} f_{\pm}(n)^{2} - \l(\frac{p}{L}\r)^{\a}\sum_{\substack{n = 1\\ p \nmid n}}^{\infty}f_{\pm}(n)^{2}n^{\a}\\
			&= \prod_{\substack{M < q \leq L\\ q \not= p}}\l( 1 + f_{\pm}(q)^{2} \r) - \l(\frac{p}{L}\r)^{\a}\prod_{\substack{M < q \leq L\\ q \not= p}}\l(1 + f_{\pm}(q)^{2}q^{\a}\r)\\
			&= \frac{1}{1 + f_{\pm}(p)^{2}} \prod_{M < q \leq L}\l( 1 + f_{\pm}(q)^{2} \r) - \l(\frac{p}{L}\r)^{\a} \frac{1}{1 + f_{\pm}(p)^{2}p^{\a}} \prod_{M < q \leq L}\l(1 + f_{\pm}(q)^{2}q^{\a}\r)\\
			&= \l( \frac{1}{1 + f_{\pm}(p)^{2}} - \frac{\l( p/ L \r)^{\a}}{1 + f_{\pm}(p)^{2}p^{\a}}\prod_{M < q \leq L}\frac{1 + f_{\pm}(q)^{2} q^{\a}}{1 + f_{\pm}(q)^{2}} \r)
			\times \prod_{M < q \leq L}\l( 1 + f_{\pm}(q)^{2} \r).
		\end{align}
		Since the estimate $f_{\pm}(p)^{2} p^{\a} \ll Q / \log{L} \asymp h$ holds and $f_{\pm}$ is supported on square-free and $M < p \leq L$, this is also equal to
		\begin{align}
			\l\{ 1 - \l( 1 + O\l( \frac{Q}{\log{L}} \r) \r) \l( \frac{p}{L} \r)^{\a} \prod_{M < q \leq L}\frac{1 + f_{\pm}(q)^{2} q^{\a}}{1 + f_{\pm}(q)^{2}}
			+ O\l( \frac{Q}{\log{L}} \r) \r\} \sum_{n = 1}^{\infty}f_{\pm}(n)^{2}.
		\end{align}
		Observe that
		\begin{align}
			\prod_{M < q \leq L}\frac{1 + f_{\pm}(q)^{2} q^{\a}}{1 + f_{\pm}(q)^{2}}
			&= \prod_{M < q \leq L}\l(1 + \frac{f_{\pm}(q)^{2} (q^{\a} - 1)}{1 + f_{\pm}(q)^{2}}\r)\\
			&= \exp\l( \l( 1 + O\l( \frac{Q}{\log{L}} \r) \r)\sum_{M < q \leq L} f_{\pm}(q)^{2} (q^{\a} - 1) \r).
		\end{align}
		Routine calculations using the prime number theorem and partial summation show that
		\begin{align}
			\sum_{M < p \leq L} \frac{\sin^{2}(\tfrac{h}{2}\log{p})}{p^{1 + \a} h \log{p}}
			= \frac{\pi}{2} \vp_{2}(\sL; \k) + O\l( \sqrt{h \log\log{L}} \r),
		\end{align}
		\begin{align}
			\sum_{M < q \leq L} f_{\pm}(q)^{2} (q^{\a} - 1)
			= \l( 1 + O\l( \sqrt{h \log\log{L}} \r) \r)\frac{\pi}{4} Q \vp_{3}(\sL; \k),
		\end{align}
		and that
		\begin{align}
			\sum_{M < p \leq L} \frac{\sin^{2}(\tfrac{h}{2}\log{p})}{p h \log{p}}
			= \frac{\pi}{2} \vp(\sL) + O\l( \sqrt{h \log\log{L}} \r),
		\end{align}
		where $\sL = \frac{h}{2\pi} \log{L}$, and $\vp$ is as in the statement of Theorem \ref{LGZMO}, and $\vp_{2}$, $\vp_{3}$ are by
		\begin{align}
			\label{def_vp2}
			\vp_{2}(\sL; \k)
			&\ceq \int_{0}^{\sL} \l(\frac{\sin(\pi u)}{\pi u}\r)^{2} e^{-2\pi \k u} \sd u,\\
			\label{def_vp3}
			\vp_{3}(\sL; \k)
			&\ceq \int_{0}^{\sL} \frac{\sin^{2}(\pi u)}{(\pi u)^{3}} \frac{e^{2\pi \k u} - 1}{e^{4 \pi \k u}} \sd u.
		\end{align}
		Therefore, quantity \cref{ELBRTsip1Rankintrk} is
		\begin{align}	\label{p1ELBRTsiRktrk}
			\geq \l\{ \vp_{2}(\sL; \k) - \l( 1 + E_{1} \r)\vp(\sL)\exp\l( - \k h \log{L} + (1 + E_{2})\frac{\pi}{4} Q \vp_{3}(\sL; \k) \r) + E_{3} \r\} \sqrt{Q}.
		\end{align}
		Here, the error terms $E_{1}, E_{2}, E_{3}$ satisfy $E_{1} \ll h$, and $ E_{2}, E_{3} \ll \sqrt{h \log\log{L}}$.
		By the choice of $Q$, we find that this lower bound is
		\begin{align}
			\label{PLB}
			\l\{ \vp_{2}(\sL; \k) - (1 + E_{1})\vp(\sL)\exp\l(- \k (y - (1 - y) E_{2}) h \log{L} \r) + E_{3} \r\} \sqrt{\frac{4 \k (1 - y)}{\pi \vp_{3}(\sL; \k)}} \sqrt{h \log{L}}.
		\end{align}
		Noting the choices of $y$ and $\k$, we see that for $l = 1 / \k \sqrt{\log(1 / \k)}$
		\begin{align}
			\vp_{2}(\sL; \k)
			&= \int_{0}^{l} \l( \frac{\sin(\pi u)}{\pi u} \r)^{2} \l( 1 + O\l( \k u \r) \r) \sd u
			+ O\l(\int_{l}^{\sL} \frac{\d u}{u^{2}} \r)\\
			&= \int_{0}^{\infty} \l( \frac{\sin(\pi u)}{\pi u} \r)^{2} \sd u
			+ O\l( \k \log{l} + \frac{1}{l} \r)
			= \frac{1}{2} + O\l( \k \log(1 / \k) \r),
		\end{align}
		and that
		\begin{align}
			\vp_{3}(\sL; \k)
			&= 2 \k \int_{0}^{l} \l( \frac{\sin(\pi u)}{\pi u} \r)^{2} \l( 1 + O\l( \k u \r) \r) \sd u
			+ O\l(\int_{l}^{\sL} \frac{\d u}{u^{3}} \r)\\
			&= 2 \k \int_{0}^{\infty} \l( \frac{\sin(\pi u)}{\pi u} \r)^{2} \sd u
			+ O\l( \k^{2} \log{l} + \frac{1}{l^{2}} \r)
			= \k \l(1 + O\l( \k \log(1 / \k) \r) \r).
		\end{align}
		Hence, \cref{PLB} is
		\begin{align}
			\l\{1 + O\l(\sqrt{h \log\log{T}} + \min\l\{ \sqrt{\frac{\log^{3}(h \log{T})}{h \log{T}}}, \frac{(\log\log{T})^{3/2}}{h^{3/2} \log{T}} \r\}\r) \r\}\sqrt{\frac{h}{\pi} \log{L}},
		\end{align}
		which completes the proof of \cref{ELBRTsip1Rankintrk}.
	\end{proof}

	The above proof gives a good lower bound of the ratio of resonators.
	In particular, we obtain the following theorem by combining the lower bound with \cref{LGZMO}.

	\begin{theorem}
		Let $\mathscr{A}$ be the set of arithmetic functions such that the value at one is not equal zero.
		For any large $L$ and for $h \in [C / \log{L}, c / \log\log{L}]$ with positive constants $C$ large and $c$ small, we have
		\begin{align}
			&\sup_{f \in \mathscr{A}} \l\{\pm \Re \sum_{km \leq L} \frac{\Lam(k)}{\sqrt{k} \log{k}} \sin(\tfrac{h}{2}\log{k}) f(m) \bar{f(k m)} \bigg/ \sum_{n \leq L} |f(n)|^{2}\r\}
			= (1 + E) \sqrt{\frac{h}{\pi} \log{L}},
		\end{align}
		where
		\begin{align}
			E
			\ll \min\l\{\sqrt{\frac{\log^{3}(h \log{L})}{h \log{L}}}, \frac{(\log\log{L})^{3/2}}{h^{3/2} \log{L}}\r\} + \sqrt{h \log\log{L}}.
		\end{align}
	\end{theorem}

	\begin{proof}
		The lower bound has already shown in \cref{ELBRTsip1Rankintrk}.
		The upper bound can be also proved by \cref{LGZMO}.
		Actually, we use \cref{LGZMO} with $W = \frac{4}{\pi} h \log L$ to obtain that
		\begin{align}
			&\bigg|\Re \sum_{km \leq L} \frac{\Lam(k)}{\sqrt{k} \log{k}} \sin(\tfrac{h}{2}\log{k}) f(m) \bar{f(k m)}\bigg| \bigg/ \sum_{n \leq L} |f(n)|^{2}\\
			&\leq \frac{\vp(\frac{h}{2\pi} \log L)}{2} \sqrt{\frac{4}{\pi} h \log T} + \frac{h \log L}{\pi \sqrt{\frac{4}{\pi} h \log L}} + O(h\sqrt{h \log L})
			\leq \l( 1 + O(h) \r)\sqrt{\frac{h}{\pi} \log{L}}
		\end{align}
		since $\vp(x) \leq 1/2$.
		Thus, we also obtain the upper bound.
	\end{proof}

\section{\textbf{Proof of \cref{LGZ}}}\label{Proof of Theorem 2}
	Let $r$ be a sufficiently large positive integer.
	For any $\theta' > 0$ and any large $T \geq T_{0}(r, \theta')$, we have
	\begin{align}
		\inf_{T \leq t \leq 2T}S(t + 2\pi r \theta' / \log T) - S(t)
		&\leq -\l( 1 + O\l( \frac{(\log r)^{3/2}}{r^{1/2}} \r) \r)\sqrt{2 r \theta'}\\
		&= -r \l( \sqrt{\frac{2}{r}\theta'} + O\l( \frac{(\log r)^{3/2}}{r} \r) \r)
	\end{align}
	by \cref{ORSi}.
	Therefore, if $\theta$ is chosen as
	\begin{align}
		\theta = 1 + \frac{\sqrt{2}}{\sqrt{r}} - C_{1} \frac{(\log r)^{3/2}}{r}
	\end{align}
	with $C_{1}$ a sufficiently large absolute positive constant, then the inequality
	\begin{align}
		\inf_{T \leq t \leq 2T}\l\{S(t + 2\pi r \theta' / \log T) - S(t)\r\}
		\leq -b
	\end{align}
	holds for any large $T$, where
	\begin{align}
		\theta' = 1 + \frac{\sqrt{2}}{\sqrt{r}} - \frac{C_{1}}{2} \frac{(\log r)^{3/2}}{r} > \theta,
	\end{align}
	and
	\begin{align}
		b = r(\theta' - 1) + \frac{C_{1}}{3} (\log r)^{3/2} > r(\theta' - 1).
	\end{align}
	Hence, by \cref{RStGZ}, we obtain the inequality of $\lam_{r}$ in \cref{INELGZ}.
	Similarly, we can prove the inequality of $\mu_{r}$.
	\qed

\section{\textbf{Proof of \cref{LGZMO}}}	\label{Sec:Proof_UBMO}

	Let $h$ be an arbitrary positive number, and let $L$ be large.
	Let $W$ be an arbitrary positive number.
	Define $g(n) = \sqrt{W} / h \sqrt{n} \log{n}$.
	Then we adapt Soundararajan's method, which uses an inequality similar to $2|f(m) \sin(\frac{h}{2} \log{k}) f(k m)| \leq |f(m)|^{2} \sin^{2}(\frac{h}{2}\log{k}) g(k) + |f(k m)|^{2} / g(k)$, to obtain
	\begin{align}
		&\bigg|\Re\frac{2}{\pi} \sum_{km \leq L} \frac{\Lam(k)}{\sqrt{k} \log{k}} \sin(\tfrac{h}{2}\log{k}) f(m) \bar{f(k m)}\bigg|\\
		&\leq \frac{1}{\pi} \sum_{2 \leq k \leq L} \frac{\Lam(k)}{\sqrt{k} \log{k}} \sum_{k m \leq L} \l( |f(m)|^{2} \sin^{2}(\tfrac{h}{2}\log{k}) g(k) + \frac{|f(k m)|^{2}}{g(k)}\r)\\
		&= \frac{1}{\pi} \sum_{n \leq L} |f(n)|^{2} \l( \sqrt{W}\sum_{p^{a} \leq L / n}  \frac{\sin^{2}(\frac{h}{2}\log{p^{a}})}{l^{2} p^{a} h \log{p}}
		+ \sum_{k \mid n} \frac{h \Lam(k)}{\sqrt{W}}\r)\\
		&= \frac{1}{\pi} \sum_{n \leq L} |f(n)|^{2} \l( \sqrt{W} \sum_{p^{a} \leq L / n}  \frac{\sin^{2}(\frac{h}{2}\log{p^{a}})}{a^{2} p^{a} h \log{p}} + \frac{h \log{n}}{\sqrt{W}}\r).
	\end{align}
	Routine calculations using the prime number theorem and partial summation show that
	\begin{align}
		\sum_{p^{a} \leq L / n}  \frac{\sin^{2}(\frac{h}{2}\log{p^{a}})}{a^{2} p^{a} h \log{p}}
		&= \sum_{p \leq L / n}  \frac{\sin^{2}(\frac{h}{2}\log{p})}{p h \log{p}} + O(h)\\
		&= \int_{2}^{L / n} \frac{\sin^{2}(\frac{h}{2}\log{\xi})}{\xi h (\log{\xi})^{2}}\sd\xi + O(h)
		= \frac{\pi}{2} \vp\l(\tfrac{h}{2\pi}\log(L / n)\r) + O(h).
	\end{align}
	Therefore, it holds that
	\begin{align}
		&\frac{1}{\pi} \sum_{n \leq L} |f(n)|^{2} \l\{ \sqrt{W} \sum_{p^{a} \leq L / n}  \frac{\sin^{2}(\frac{h}{2}\log{p^{a}})}{l^{2} p^{a} h \log{p}} + \frac{h \log{n}}{\sqrt{W}}\r\}\\
		&\leq \frac{1}{\pi} \l( \max_{n \leq L} \l\{ \sqrt{W} \frac{\pi}{2} \vp\l( \tfrac{h}{2\pi} \log(L / n) \r) + \frac{h \log{n}}{\sqrt{W}}\r\} + O(h \sqrt{W})\r) \sum_{n \leq L} |f(n)|^{2}\\
		&\leq \l( \max_{1 \leq l \leq L} \l\{ \frac{\sqrt{W}}{2} \vp\l( \tfrac{h}{2\pi} \log(L / l) \r) + \frac{h \log{l}}{\pi \sqrt{W}}\r\} + O(h \sqrt{W})\r) \sum_{n \leq L} |f(n)|^{2}.
	\end{align}
	Hence, we have
	\begin{align}
		&\bigg|\Re\frac{2}{\pi} \sum_{km \leq L} \frac{\Lam(k)}{\sqrt{k} \log{k}} \sin(\tfrac{h}{2}\log{k}) f(m) \bar{f(k m)}\bigg|\\
		&\leq \l( \max_{1 \leq l \leq L} \l\{ \frac{\sqrt{W}}{2} \vp\l( \tfrac{h}{2\pi} \log(L / l) \r) + \frac{h \log{l}}{\pi\sqrt{W}}\r\} + O(h \sqrt{W})\r) \sum_{n \leq L} |f(n)|^{2}.
	\end{align}
	Since the parameter $W$ is arbitrary, we have \cref{LGZMO1}.
	\qed

\section{\textbf{Limitations of $\lam_{1}$ and $\mu_{1}$ deduced from \cref{LGZMO}}}	\label{Sec:limitation}
	If $\tau(\xi; f) > 1$ for any $L \leq T$, any arithmetic function $f$, and any $\xi \geq \xi_{0}$, then the bound $\lam_{1} \geq \xi_{0}$ becomes the limitation for the Montgomery-Odlyzko method.
	Note that the right hand side of the inequality in \cref{LGZMO} is increasing for $L$, and hence we may assume $L = T$ in the following argument.
	Using \cref{LGZMO} with $l = T^{x}$, $W = 22.6$, we have
	\begin{align}
		\tau(\xi; f)
		\geq \xi - \max_{0 \leq x \leq 1}\l\{\frac{\sqrt{W}}{2} \vp\l( \xi (1 - x) \r) + \frac{2 \xi x}{\sqrt{W}}\r\} - o(1)
	\end{align}
	for any arithmetic function $f$ that is not identically zero.
	The right hand side exceeds one when $\xi \geq \xi_{0} = 3.022$.
	From this observation, we deduce that the limitation of $\lam_{1}$ for the Montgomery-Odlyzko method is $\lam_{1} \geq 3.022$.
	Similarly, we conclude that the limitation of $\mu_{1}$ for the Montgomery-Odlyzko method is $\mu_{1} \leq 0.508$ by using \cref{LGZMO} with $W = 4.9$.

	\begin{remark}
		We should mention the difference between our work and Goldston, Trudgian, Turnage-Butterbaugh \cite{GTTT2023}.
		The works \cite{CGG1983}, \cite{GTTT2023} used the inequality $|z w| \leq (|z|^{2} + |w|^{2})/2$ for $w, z \in \CC$, partial summation, and the inequality $|\sin x| \leq |x|$ for $x \in \RR$.
		On the other hand, we avoid using the final inequality and instead derive the result by a direct analysis of the sine function.
	\end{remark}

	\begin{acknowledgment*}
		The authors would like to thank Professor Timothy Trudgian for providing us with valuable comments.
		The first author is supported by JSPS KAKENHI Grant Number 24K16907.
		The second author is supported by JSPS KAKENHI Grant Number 25K17245.
		The third  author is supported by Grant-in-Aid for JSPS Research fellow Grant Number 24KJ1235.
	\end{acknowledgment*}

\end{document}